\documentclass[twoside,11pt]{article}
\usepackage{amsmath,amssymb,amsthm,amsfonts,mathrsfs,wasysym,latexsym,times,lineno, subfigure,color}
\usepackage{amsmath,amsfonts}
\usepackage{amsthm}
\usepackage{graphicx}
\usepackage{color}
\usepackage{multicol}
\usepackage{amsmath}
\allowdisplaybreaks[4]

\newtheorem{thm}{Theorem}

\newtheorem{definition}{Definition}
\newtheorem{corollary}{Corollary}

\newtheorem{lemma}[thm]{Lemma}
\newtheorem{remark}{Remark}

\newtheorem{theorem}[thm]{Theorem}
\newtheorem{proposition}[thm]{Proposition}

\title{A sharp threshold of propagation connectivity for mixed random hypergraphs}

\author{Guangyan Zhou$^{1,a}$,\quad Bin Wang$^{b}$,\quad Ke Xu$^{c}$ \\
\footnotesize $^{a}$Department of Mathematics, Beijing Technology and Business University,  China\\
\footnotesize $^{b}$ Key Laboratory of Random Complex Structures and Data Science, Academy of Mathematics and \\
\footnotesize Systems Science, Chinese Academy of Sciences, China\\
\footnotesize $^{c}$ State Key Laboratory of Software Development Environments, Beihang University, China\\
}
\date{}
\begin{document}
\maketitle
\begin{abstract}
This paper studies the propagation connectivity of a random hypergraph $\mathbb{G}$ containing both 2-edges and 3-hyperedges. We find an exact threshold of the propagation connectivity of $\mathbb{G}$: If $I_{\epsilon,r}<-1$, then $\mathbb{G}$ is not propagation connected with high probability; while if $I_{\epsilon,r}>-1$, then $\mathbb{G}$ is propagation connected with high probability, where $I_{\epsilon,r}$ is a constant dependent on the parameters of 2 and 3-edge probabilities.


\end{abstract}
\footnotetext[1]{Corresponding author. Email: zhouguangyan@btbu.edu.cn, wangbin@amss.ac.cn, kexu@nldse.buaa.edu.cn.

 Research supported
by National Natural Science Fund of China (Grant No.
11501017, 11626039, 61702019).}
\section{Introduction }
The study of phase transition phenomena of constraint satisfaction problems is one of the most interesting topics  at the intersection of mathematics, statistical physics and computer science. In mathematics, since the seminal work of Erd\H{o}s and R\'{e}nyi \cite{ER}, identifying the thresholds for different properties has been a major task in the theory of random graphs and hypergraphs. Throughout the years, various types of graphs have been studied by graph theorists, such as planar graphs, routing networks and computational graphs that are used in designing algorithms or simulations, etc. Graph theory has emerged as a primary tool for detecting numerous hidden structures in various information networks, including internet graphs, social networks, or more generally, any graph representing relations in massive data sets.
Nowadays, more important properties of random graph have been found, and the picture of the evolution of the random graph is fairly complete.

Connectivity is perhaps the most basic property of graphs and hypergraphs, and is also a fundamental combinatorial problem. For random graphs, the definition of connectivity  is quite natural. For random graph $\mathbb{G}(n,p)$, which has $n$ vertices and each edge appears with probability $p=\frac cn$  ($c$ is a constant), there exist double phase transitions  near $c=1$ where the size of the largest connected component changes twice - first from $\Theta(\ln n)$ to $\Theta( n^{2/3})$, and then from $\Theta( n^{2/3})$ to $\Theta(n)$ \cite{JLR}.   For random hypergraphs, the definition of connectivity may differs. The `standard' concept of random hypergraph connectivity (where edges are replaced by triangles) has been studied \cite{BCK,BCK2,CMS}. Propagation connectivity is one kind of connectivity that has been studied in recent years, e.g. on general hypergraphs \cite{BO09} and on 3-uniform hypergraphs \cite{DN05,COW}.
Intuitively, propagation connectivity is analogous to the growth of a network: starting from a vertex, or a small initial graph, adding a new node and a new edge at each time following certain growth rules. It can be explained in terms of a simple marking process (or branching process): at each step, if there exists an edge $e$ where all vertices have been marked except one vertex, then we mark this vertex. If there exists a marking process such that all the vertices can be marked, then this graph is propagation connected. Specifically,  Coja-Oghlan et al. \cite{COW} obtained the following results on 3-uniform hypergraphs.

\begin{theorem}\cite{COW}\label{COW}
Suppose that $p=\frac{r}{n\ln n}$ for a constant $r>0$.\\
(1) If $r<0.16$, then $\mathbf{H}(n,p)$ fails to be propagation connected w.h.p. $^1$\\
(2) If $r>0.25$, then $\mathbf{H}(n,p)$ is propagation connected w.h.p.
\end{theorem}
\footnotetext[1]{We say a sequence of events $\xi_n$ occurs with high probability (w.h.p.) if $\lim_{n\rightarrow\infty}Pr[\xi_n]=1$.}

In this paper, we study the phase transition window of random graph whose edge density stays below the threshold mentioned above (i.e. $c<1$) by adding 3-hyperedges into it. Specifically, in this paper we consider the propagation connectivity of a random hypergraph $\mathbb{G}=\mathbb{G}(n,p_2,p_3)$, which is composed of both 2-edges and 3-hyperedges on $n$ vertices. Each possible 2-edge on exactly two vertices is included in $\mathbb{G}(n,p_2,p_3)$ with probability $p_2$, and each possible 3-hyperedge on exactly three vertices is included in $\mathbb{G}(n,p_2,p_3)$ with probability $p_3$. Throughout this paper, we let
\begin{equation}\label{defp2p3}
p_2=\frac{1-\epsilon}{n},p_3=\frac{r}{n\ln n},
\end{equation}where parameters $0<\epsilon<1,r>0$ are constants.

 Using probabilistic methods, we study the propagation process starting from a random vertex, and present our most valuable and technical contribution in Theorem \ref{theorem}  that there exists a sharp phase transition of propagation connectivity of hypergraph $\mathbb{G}(n,p_2,p_3)$. A special case of Theorem \ref{theorem} closes the gap left by Coja-Oghlan et al. \cite{COW} in Theorem \ref{COW}.

The rest of the paper is organized as follows. In Section 2, we introduce the definition of propagation connectivity in hypergraph $\mathbb{G}(n,p_2,p_3)$. In Section 3, we present our main results on the threshold of propagation connectivity of the random hypergraph $\mathbb{G}(n,p_2,p_3)$. The proof of our main theorem is put in Section 4. In particular, the main technique in this paper lies heavily on Markov chain techniques which give relatively precise estimates of some crucial probabilities. Section 5 draws a conclusion.


\section{Propagation connectivity}

\begin{definition}(Propagation connectivity).
Let $H=(V,E)$ be a hypergraph on $n=|V|$ vertices, and the edge set $E$ contains both edges of length 2 and edges of length 3. If there exits a propagation sequence $e_1,e_2,...,e_{n-1}\in E$, such that for any integer $1\le l\le n-2$ we have $|e_{l+1}\cap\bigcup_{i=1}^le_i|=|e_{l+1}|-1$, then we say $H$ is propagation connected.
\end{definition}

To explain this concept, we show that the propagation sequence of a hypergraph $\mathbb{G}(n,p_2,p_3)$  is in fact  a certain Markov additive process. During the propagation process, vertices will be labeled as active, inactive and unexplored. We denote by $V$  the vertex set of $\mathbb{G}$, $\mathcal{Y}_t$ the set of active vertices at time $t$, and $\mathcal{D}_t$ the set of inactive vertices at time $t$. At the beginning, $\mathcal{Y}_0=\{v_0\}$, $\mathcal{D}_0=\emptyset$. At time $t$ ($t=0,1,2...$), we pick an active vertex $v_t$ and check all unexplored vertices that can be connected to $v_t$. For a random unexplored vertex $u$, if there exists a 2-edge which connects $v_t$ and $u$, or there exists a 3-edge which connects $v_t,u$ and an inactive vertex $\omega\in\mathcal{D}_{t}$, then $u$ becomes active and at the same time, $v_t$ becomes inactive:
\begin{align*}
\mathcal{Y}_{t+1}&=\mathcal{Y}_t\cup\{u:u\text{ becomes active}\}\setminus\{v_t\},\\
\mathcal{D}_{t+1}&=\mathcal{D}_t\cup\{v_t\}.
\end{align*}
Intuitively, each step increases exactly one vertex to $\mathcal{D}_t$ which can be considered as the formal set of vertices propagation connected from $v_0$. Let $Y_t=|\mathcal{Y}_t|$ be the total number of vertices that were ``active" at time $t$. Denote by $Z_t$  the total number of new vertices $u$ that becomes active at step $t$. Once the set $\mathcal{Y}_t$ becomes empty (i.e. $Y_t=0$), the process terminates. We denote $C_{v_0}$ the component  whose vertices are propagation connected from $v_0$.  It is easy to see that the first hitting time $T_{v_0}=\inf\{t:Y_t=0\}$ is exactly the size of $C_{v_0}$.

The process above is in fact a Markov additive process, where the distribution of the increment $Z_t$ is only dependent on $Y_{t}$ and $t$, which means
\begin{align*}
\mathbf{Pr}[Z_t\big|Y_1=y_1,..,Y_t=y_t]=\mathbf{Pr}[Z_t\big|Y_t=y_t].
\end{align*}

It is easy to see that
\begin{equation*}
   \left\{
 \begin{aligned}
 &Y_0=1,\\
 &Y_{t+1}=Y_{t}+Z_t-1.
 \end{aligned}
 \right.
 \end{equation*}
Note that at time $t$, given an active vertex $v_t$, a random unexplored vertex $\omega$ is not connected to $v_t$ through a 2-edge with probability $1-p_2$, and is not connected to $v_t$ through a 3-edge with probability $(1-p_3)^{t}$, since there are $t$ inactive vertices at time $t$. On the other hand, the total number of unexplored vertices is $n-t-Y_{t}$. Thus the conditional distribution of $Z_t$  on $Y_t$ follows a binomial distribution as
\begin{equation}\label{pt}
Z_t|Y_{t}\sim \mathbf{B}[n-t-Y_{t},p(t)],\text{ where }p(t)=1-(1-p_2)(1-p_3)^{t}.
\end{equation}
For simplicity, we define three functions which will be used later.
\begin{align}\label{funlamda}
\nonumber\lambda(x)&=1-\epsilon+rx,\\
\lambda_1(x)&=1-\epsilon+\frac r2x,\\
\nonumber\lambda_2(x)&=\frac{1-\epsilon}{2}+\frac r6x.
\end{align}

\section{Main results}

In the following we present the most important theorem of this paper, a sharp threshold of propagation connectivity of random hypergraph $\mathbb{G}(n,p_2,p_3)$.
\begin{theorem}\label{theorem}
Let $p_2=\frac{1-\epsilon}{n}$, $p_3=\frac{r}{n\ln n}$, where $0<\epsilon<1,r>0$ are constants. The random hypergraph $\mathbb{G}(n,p_2,p_3)$ has the following properties:
\begin{itemize}
  \item If $I_{\epsilon,r}<-1$, then w.h.p. $\mathbb{G}$ is not propagation connected;
  \item If $I_{\epsilon,r}>-1$, then w.h.p. $\mathbb{G}$ is propagation connected,
\end{itemize}
where
\begin{equation}
I_{\epsilon,r}=\int_{0}^{\frac{\epsilon}{r}}\big[1-\lambda(\omega)+\ln \lambda(\omega)\big]d\omega=-\frac{1}{r}\left[\epsilon-\frac{\epsilon^2}{2}+(1-\epsilon)\ln(1-\epsilon)\right].
\end{equation}
\end{theorem}
\begin{remark}
It is noteworthy that a special case of $p_2=0$ in Theorem \ref{theorem} corresponds to the 3-uniform hypergraph, and our results not only match the results of Theorem 1 obtained by \cite{COW} but also close the gap of threshold for connectivity left by them.
\end{remark}


\section{Proof of Theorem \ref{theorem}}

Below we explain the main idea behind the proof of our main results.
In Section 4.1, we prepare two lemmas which will be used later in our proof. In Section 4.2, we show that w.h.p. there exists no propagation component of intermediate size between $O(\ln n)$ and $n-1$. In Section 4.3,  we prove that if $I_{\epsilon,r}<-1$, then w.h.p. the largest propagation component size is at most $O(\ln n)$, thus w.h.p. $\mathbb{G}$ can not be propagation connected.
Based on the result that there is no existence of propagation component with intermediate size, we know that once a propagation process survives after $O(\ln n)$ steps, then it is possible that this process will continue to time $n$, which means that the entire hypergraph will be propagation connected by this process. Moreover, in Section 4.3.3, we compare the results of \cite{COW} with a special case of our results, and show that our results actually close the gap left by them.
Finally, in Section 4.4, we prove that if $I_{\epsilon,r}>-1$, then $\mathbb{G}$ is w.h.p. propagation connected.
\subsection{Some lemmas}

We prepare two lemmas which will be used for the proof of our main results later.
\begin{lemma}
Suppose random variables $X_1,X_2,...,X_n$ are independently and identically distributed, denote two events as
\begin{align*}
A&=\{X_1\geq0,X_1+X_2\geq0,...,X_1+X_2+...+X_n\geq 0\},\\
B&=\{X_1+X_2+...+X_n\geq0\}.
\end{align*}
Then
$$\frac{\mathbf{Pr}(B)}n\leq \mathbf{Pr}(A)\leq \mathbf{Pr}(B).$$
\end{lemma}

\begin{proof}
The second inequality on the right hand side holds trivially. So we will only prove the inequality on the left side.

For any $k\in N$, let the subscript of $X_k$ module $n$, i.e. $X_{n+k}=X_k$, so that all the subscripts are at most $n$. Define $S_k=X_1+...+X_k$. Consider the following events:
$$A_k=\{X_k\geq0,X_k+X_{k+1}\geq0,...,X_k+X_{k+1}+...+X_{k+n-1}\geq 0\},k=1,...,n.$$
By symmetry, it is easy to see that
$$\mathbf{Pr}(A)=\mathbf{Pr}(A_1)=...=\mathbf{Pr}(A_n).$$
Suppose the event $B$ happens, i.e. $S_n\geq0$, and suppose $S_t$ ($0\leq t\leq n-1$) is the minimum one of the set $\{S_0\equiv0,S_1,...,S_{n-1}\}$. Then for all $k=1,2,...,n$, if $t+k\leq n-1$ we have $S_{t+k}\geq S_t$; if $t+k\geq n$, we have $S_{t+k}=S_{t+k-n}+S_n\geq S_t$.  Therefore it always holds that
$$X_{t+1}+X_{t+2}+...+X_{t+k}=S_{t+k}-S_t\geq0,$$
 which means that the event $A_{t+1}$ happens, thus
 $$\mathbf{Pr}(B)\leq\mathbf{Pr}(A_1)+\mathbf{Pr}(A_2)+...+\mathbf{Pr}(A_n)=n\mathbf{Pr}(A).$$
\end{proof}

\begin{definition}
Let the distribution function of a random variable $X$ be $F(x)=\mathbf{Pr}[X\leq x]$, and denote $\overline{F}(x)=\mathbf{Pr}[X> x]=1-F(x)$. We say that $F\succ G$, if two distribution functions $F$ and $G$ satisfying $F(t)\leq G(t)$ (or $\overline{F}(t)>\overline{G}(t)$) for all $t\in R$.
\end{definition}

 Take the binomial distribution for example, if $n_1\geq n_2$ then  $\mathbf{B}(n_1,p_1)\succ\mathbf{B}(n_2,p_1)$; if $p_1\geq p_2$, then $\mathbf{B}(n_1,p_1)\succ\mathbf{B}(n_1,p_2)$.

With this notation, we consider two Markov additive processes, and have the following results which will be important to our argument.

\begin{lemma}\label{lemmadayu}
Suppose two stochastic processes $\{S_k\}$ and $\{T_k\}$ satisfy that
$$S_{k+1}=S_{k}+X_k-q,T_{k+1}=T_{k}+Y_k-q, $$
where $S_0=T_0=b\in Z$, $q$ is a nonnegative integer, $X_k|S_k\sim F_{k,S_k}$ and $Y_k|T_k\sim G_{k,T_k}$. If the distribution functions $F_{k,S_k}$ and $G_{k,T_k}$ take values of non-negative integers, and satisfying
\begin{enumerate}
  \item For any $k\in N,z\in Z,l\in R$, it holds that $G_{k,z}(l)\leq G_{k,z-1}(l+1)$;
  \item There exists an integer $M\geq b$ such that for any $z\leq M-kq$, it holds that $F_{k,z}\succ G_{k,z}.$
\end{enumerate}
Then for any sequence $\{l_k|l_k\leq M-kq\}$,
$$\mathbf{Pr}[S_1>l_1,S_2>l_2,...,S_n>l_n]\geq \mathbf{Pr}[T_1>l_1,T_2>l_2,...,T_n>l_n].$$
\end{lemma}

\begin{proof}
We prove this claim by induction on $n$.
For simplicity, denote
\begin{align*}
A_n=\{S_1>l_1,S_2>l_2,...,S_n>l_n\};\\
B_n=\{T_1>l_1,T_2>l_2,...,T_n>l_n\}.
\end{align*}

If $n=1$, and $X_0\thicksim F_{0,b}, Y_0\thicksim G_{0,b}, F_{0,b}\succ G_{0,b}$, then the claim holds since
\begin{align*}
\mathbf{Pr} [A_1]&=\mathbf{Pr}[S_1\geq l_1]=\mathbf{Pr}[X_0>l_1+q-b]=\overline{F_{0,b}}(l_1+q-b)\\
&\geq\overline{G_{0,b}}(l_1+q-b)=\mathbf{Pr} [B_1].
\end{align*}

Suppose the results hold for $n$, then we consider the case of $n+1$.

If $l_{n+1}\leq l_n-q$, then $S_n>l_n$ implies that $S_{n+1}=S_n+X_n-q\geq l_n-q\geq l_{n+1}$, and the result holds obviously:
$$\mathbf{Pr}[A_{n+1}]=\mathbf{Pr}[A_n]\geq\mathbf{Pr}[B_n]=\mathbf{Pr}[B_{n+1}].$$

Now we consider $l_{n+1}> l_n-q$, denote
\begin{align*}
A_{n,l}&=\{S_1>l_1,...,S_{n-1}>l_{n-1},S_n=l\};\\
\widetilde{A}_{n,l}&=\{S_1>l_1,...,S_{n-1}>l_{n-1},S_n>l\};\\
B_{n,l}&=\{T_1>l_1,...,T_{n-1}>l_{n-1},S_n=l\};\\
\widetilde{B}_{n,l}&=\{T_1>l_1,...,T_{n-1}>l_{n-1},T_n>l\}.
\end{align*}
In fact, $\widetilde{A}_{n,l}$ and $\widetilde{B}_{n,l}$ can be written as:
\begin{align*}
\widetilde{A}_{n,l}=\bigcup_{i=l+1}^{\infty}A_{n,i},
\widetilde{B}_{n,l}=\bigcup_{i=l+1}^{\infty}B_{n,i},
\end{align*}

The estimate of $\mathbf{Pr}[A_{n+1}]$ can be derived as follows.
\begin{align*}
&\mathbf{Pr}[A_{n+1}]=\mathbf{Pr}[S_1>l_1,...,S_{n}>l_{n},S_{n+1}>l_{n+1}]\\
=&\sum_{z>l_n}\mathbf{Pr}[A_{n,z}]\mathbf{Pr}[X_n|_{S_n=z}>l_{n+1}-z+q]\\
=&\sum_{l_n<z\leq l_{n+1}+q}\mathbf{Pr}[A_{n,z}]\overline{F}_{n,z}(l_{n+1}-z+q)+\sum_{z>l_{n+1}+L}\mathbf{Pr}[A_{n,z}]\\
=&\sum_{l_n<z\leq l_{n+1}+q}\mathbf{Pr}[A_{n,z}]\overline{F}_{n,z}(l_{n+1}-z+q)+\mathbf{Pr}[\widetilde{A}_{n,l_{n+1}+q}]\\
\geq&\sum_{l_n<z\leq l_{n+1}+q}\mathbf{Pr}[A_{n,z}]\overline{G}_{n,z}(l_{n+1}-z+q)+\mathbf{Pr}[\widetilde{A}_{n,l_{n+1}+q}]\\
=&\sum_{l_n<z\leq l_{n+1}+q}\Big(\mathbf{Pr}[\widetilde{A}_{n,z-1}]-\mathbf{Pr}[\widetilde{A}_{n,z}]\Big)
\overline{G}_{n,z}(l_{n+1}-z+q)+\mathbf{Pr}[\widetilde{A}_{n,l_{n+1}+q}]\\
=&\sum_{l_n<z\leq l_{n+1}+q+1}\mathbf{Pr}[\widetilde{A}_{n,z-1}]\overline{G}_{n,z}(l_{n+1}-z+q)-\sum_{l_n<z\leq l_{n+1}+q}\mathbf{Pr}[\widetilde{A}_{n,z}]\overline{G}_{n,z}(l_{n+1}-z+q)\\
=&\sum_{l_n\leq z\leq l_{n+1}+q}\mathbf{Pr}[\widetilde{A}_{n,z}]\overline{G}_{n,z+1}(l_{n+1}-z-1+q)-\sum_{l_n<z\leq l_{n+1}+q}\mathbf{Pr}[\widetilde{A}_{n,z}]\overline{G}_{n,z}(l_{n+1}-z+q)\\
=&\mathbf{Pr}[\widetilde{A}_{n,l_n}]\overline{G}_{n,l_n+1}(l_{n+1}-l_n-1+q)+\\
&\sum_{l_n<z\leq l_{n+1}+q}\mathbf{Pr}[\widetilde{A}_{n,z}]\Big[\overline{G}_{n,z+1}(l_{n+1}-z-1+q)-\overline{G}_{n,z}(l_{n+1}-z+q)\Big].
\end{align*}
The inequality above is because for $l_n<z\leq l_{n+1}+q$, the fact $l_{n+1}\leq M-(n+1)q$ implies that $z\leq M-nq$. Applying condition 2 gives $\overline{F}_{n,z}(l)\geq\overline{G}_{n,z}(l)$.

Similarly, we can estimate $\mathbf{Pr}[B_{n+1}]$.
\begin{align*}
&\mathbf{Pr}[B_{n+1}]=\mathbf{Pr}[T_1>l_1,...,T_{n}>l_{n},T_{n+1}>l_{n+1}]\\
&=\sum_{z>l_n}\mathbf{Pr}[B_{n,z}]\mathbf{Pr}[Y_n|_{l_n=z}>l_{n+1}-z+q]\\
&=\sum_{l_n<z\leq l_{n+1}+q}\mathbf{Pr}[B_{n,z}]\overline{G}_{n,z}(l_{n+1}-z+q)+\sum_{z>l_{n+1}+q}\mathbf{Pr}[B_{n,z}]\\
&=\sum_{l_n<z\leq l_{n+1}+q}\Big(\mathbf{Pr}[\widetilde{B}_{n,z-1}]-\mathbf{Pr}[\widetilde{B}_{n,z}]\Big)
\overline{G}_{n,z}(l_{n+1}-z+q)+\mathbf{Pr}[\widetilde{B}_{n,l_{n+1}+q}]\\
&=\sum_{l_n<z\leq l_{n+1}+q+1}\mathbf{Pr}[\widetilde{B}_{n,z-1}]\overline{G}_{n,z}(l_{n+1}-z+q)-\sum_{l_n<z\leq l_{n+1}+q}\mathbf{Pr}[\widetilde{B}_{n,z}]\overline{G}_{n,z}(l_{n+1}-z+q)\\
&=\sum_{l_n\leq z\leq l_{n+1}+q}\mathbf{Pr}[\widetilde{B}_{n,z}]\overline{G}_{n,z+1}(l_{n+1}-z-1+q)-\sum_{l_n<z\leq l_{n+1}+q}\mathbf{Pr}[\widetilde{B}_{n,z}]\overline{G}_{n,z}(l_{n+1}-z+q)\\
&=\mathbf{Pr}[\widetilde{B}_{n,l_n}]\overline{G}_{n,l_n+1}(l_{n+1}-l_n-1+q)+\\
&\sum_{l_n<z\leq l_{n+1}+q}\mathbf{Pr}[\widetilde{B}_{n,z}]\Big[\overline{G}_{n,z+1}(l_{n+1}-z-1+q)-\overline{G}_{n,z}(l_{n+1}-z+q)\Big].
\end{align*}

Noticing that $\overline{G}_{n,l_n+1}(l_{n+1}-l_n-1+q)\geq0$, and combining condition 1 we know that $\overline{G}_{n,l_n+1}(l_{n+1}-z-1+q)\geq\overline{G}_{n,l_n}(l_{n+1}-z+q)$. If $z=l_n$ or $l_n<z\leq l_{n+1}+q$, then in fact $z\leq M-nq$. By the induction hypothesis we know that $\mathbf{Pr}[\widetilde{A}_{n,z}]\geq\mathbf{Pr}[\widetilde{B}_{n,z}]$, therefore
$$\mathbf{Pr}[A_{n+1}]\geq\mathbf{Pr}[B_{n+1}].$$

This completes our proof.

\end{proof}

\subsection{No components of intermediate size}
\begin{lemma}\label{middlesize}
There exists a constant $K_0=\max\{1,\lambda_1^{-1}(10)\}$, such that with probability approaching one as $n\rightarrow\infty$, there is no propagation connected components of size between $K_0\ln n$ and $n-1$.
\end{lemma}

\begin{proof}
Suppose the propagation component connected from vertex $v$ is $C_v$, which contains $|C_v|$ vertices. There are two facts for $C_v$: first, there is no 2-edge connecting one vertex from $C_v$ and one vertex from $V\setminus C_v$, and there is no 3-edge connecting two vertices from $C_v$ and one vertex from $V\setminus C_v$; second, since $C_v$ is propagation connected, then $C_v$ contains at least $|C_v|-1$ edges.

Let $E_v$ denote the number of edges in $C_v$.
The probability that the size of the component $C_v$ is $z$ can be upper bounded as follows.
\begin{align}\label{pro1}
\mathbf{Pr}[|C_v|=z]\leq\binom{n}{z}(1-p_2)^{z(n-z)}(1-p_3)^{\binom z2(n-z)}\mathbf{Pr}[E_v\geq z-1].
\end{align}
The stirling's formula $n!=\sqrt{2\pi n}\left(\frac ne\right)^ne^{\theta_n}$ $\left(\frac{1}{12n+1}<\theta_n<\frac{1}{12n}\right)$ implies that for all $1\leq z\leq n$,
\begin{equation}\label{nchoosez}
\binom{n}{z}\leq\frac{n^n}{z^z(n-z)^{n-z}}=\frac{n^z}{z^z}\left(1+\frac{z}{n-z}\right)^{n-z}\leq\left(\frac{ne}{z}\right)^z.
\end{equation}
Also it is straightforward that
\begin{equation}\label{p2p3}
(1-p_2)^{z(n-z)}(1-p_3)^{\binom z2(n-z)}\leq\exp\left\{-p_2z(n-z)-p_3\binom z2(n-z)\right\}.
\end{equation}
We will estimate  (\ref{pro1}) by three cases.\\

\textbf{Case 1: } $\frac{n}{2}\leq|C_v|=z\leq n-1$. We have
\begin{equation*}
\ln\binom{n}{z}=\ln\binom{n}{n-z}\leq\ln\left(\frac{ne}{n-z}\right)^{n-z}\leq(n-z)(\ln n+1).
\end{equation*}
Taking the logarithm on both sides of (\ref{pro1}) and notice that $\mathbf{Pr}[E_v\geq z-1]\leq1$, we get
\begin{align*}
\ln\mathbf{Pr}[|C_v|=z]&\leq(n-z)(\ln n+1)-(n-z)\left(\frac{1-\epsilon}{n}z+\frac{r}{n\ln n}\frac{z^2}{2}\right)\\
&=-(n-z)\frac{rz^2}{2n\ln n}(1+o(1))\\
&\leq-3\ln n,
\end{align*}
which yields
\begin{equation*}
\mathbf{Pr}[|C_v|=z]\leq n^{-3}.
\end{equation*}
\\
\textbf{Case 2: } $(\ln n)^3\leq|C_v|=z=\frac{n}{2}$. Similar to case 1, we can apply (\ref{nchoosez}) and (\ref{p2p3}) to get

\begin{align*}
\ln\mathbf{Pr}\big[|C_v|=z\big]&\leq z(\ln n-\ln z+1)-z\left((1-\epsilon)\frac{n-z}{n}+\frac{r}{2}\frac{z(n-z)}{n\ln n}\right)\\
&\leq\ln n-z\frac{r}{2}\frac{(\ln n)^3(n-n/2)}{n\ln n}\leq -z\ln n\\
&\leq-3\ln n,
\end{align*}
which yields
\begin{equation*}
\mathbf{Pr}[|C_v|=z]\leq n^{-3}.
\end{equation*}
\\

\textbf{Case 3: } $K_0\ln n\leq|C_v|=z\leq(\ln n)^3$. We rewrite $z$ as $z=K\ln n $, where $K$ is dependent on $n$. We claim that the probability of the component size $|C_v|$ being $z$ is also upper bounded by $n^{-3}$.

To begin with, suppose $C_v$ contains $E_2$ 2-edges and $E_3$ 3-edges, then  $E_2$ and $E_3$ independently follow the binomial distribution with $E_2\thicksim \mathbf{B}\left[\binom{z}{2},p_2\right]$ and $E_3\thicksim \mathbf{B}\left[\binom{z}{3},p_3\right]$. Let $E_v=E_2+E_3$ be the total number of edges in $C_v$. Obviously, For all $s\geq1$, the generating function of $E_v$ satisfies
\begin{equation}
E[s^{E_v}]\leq\big(1+p_2(s-1)\big)^{z^2/2}\big(1+p_3(s-1)\big)^{z^3/6},
\end{equation}
thus it holds for all $s\ge1$ that
\begin{equation}
\mathbf{Pr}[E_v\geq z-1]\leq\frac{E[s^{E_v}]}{s^{z-1}}\leq\frac{1}{s^{z-1}}\big(1+p_2(s-1)\big)^{z^2/2}\big(1+p_3(s-1)\big)^{z^3/6}.
\end{equation}
Now we can upper bound the probability that $C_v$ contains at least $z-1$ edges as follows.
\begin{align*}
\ln\mathbf{Pr}[E_v\geq z-1]
&\leq\frac{z^2}{2}\ln\big(1+p_2(s-1)\big)+\frac{z^3}{6}\big(1+p_3(s-1)\big)-(z-1)\ln s\\
&\leq \frac{z^2}{2}p_2(s-1)+\frac{z^3}{6}p_3(s-1)-(z-1)\ln s\\
&=(s-1)\frac{z^2}{n}\lambda_2(K)-(z-1)\ln s.
\end{align*}
Substituting $s=s_0\equiv\frac{n}{z\lambda_2(K)}\geq1$ into the above inequality, we obtain
\begin{align}\label{edge3}
\ln\mathbf{Pr}[E_v\geq z-1]\leq z-(z-1)\ln s_0\leq z-(z-1)(\ln n-\ln z-\ln \lambda_2(K)).
\end{align}

At the same time, notice that
\begin{align*}
&\ln\binom{n}{z}\leq z(\ln n-\ln z)+z,\\
&(1-p_2)^{z(n-z)}(1-p_3)^{\binom z2(n-z)}\leq\exp\left[-z\frac{n-z}n\lambda_1(K)\right]\leq-0.9z \lambda_1(K).
\end{align*}

Putting these two inequalities and (\ref{edge3}) together, we get
\begin{align*}
\ln\mathbf{Pr}[|C_v|=z]&\leq2z+\ln n-\ln z+(z-1)\ln \lambda_2(K)-0.9z\lambda_1(K)\\
&\leq z\big(2+\ln \lambda_2(K)-0.9\lambda_1(K)\big)+\ln n.
\end{align*}
Note that we require $K_0=\max\{1,\lambda_1^{-1}(10)\}$, thus for all $z=K\ln n\geq K_0\ln n$, it holds that $\lambda_1(K)\geq10$ and $\ln \lambda_2(K)-0.9\lambda_1(K)\leq-6$. Now for any $K_0\ln n\leq z\leq(\ln n)^3$ we have
\begin{align*}
\ln\mathbf{Pr}[|C_v|=z]\leq-4z+\ln n\leq-3\ln n,
\end{align*}
which yields
\begin{equation*}
\mathbf{Pr}[|C_v|=z]\leq n^{-3}.
\end{equation*}

As a consequence, putting the three cases together, we conclude that
\begin{align*}
\sum_{z=K_0\ln n}^{n-1}\mathbf{Pr}[|C_v|=z]\leq n\cdot n^{-3}=n^{-2},
\end{align*}
which completes our proof.
\end{proof}


\subsection{Subcritical regime: $I_{\epsilon,r}<-1$}

\subsubsection{Proof of the lower bound in Theorem \ref{theorem}}
To prove that  $\mathbb{G}$ is not propagation connected,  it suffices to show that almost all propagation processes terminate before $O(\ln n)$ steps. Recall that $K_0=\max\{1,\lambda_1^{-1}(10)\}$.

\begin{proposition}\label{m=1}
Let $v_0$ be any given vertex, then
\begin{equation}
\mathbf{Pr}[|C_{v_0}|\geq K_0\ln n]\leq n^{I_{\epsilon,r}+o(1)},
\end{equation}
where $C_{v_0}$ is the propagation component generated by $v_0$.
\end{proposition}

From Proposition \ref{m=1}, if we let $I_{\epsilon,r}<-1$, then
\begin{equation*}
\mathbf{Pr}[\exists\text{ a vertex }v,\text{ s.t. } |C_{v}|\geq K_0\ln n]\leq n\cdot n^{I_{\epsilon,r}+o(1)}=o(1),
\end{equation*}
which proves that w.h.p. $\mathbb{G}$ is not propagation connected.  This completes the proof of the lower bound of Theorem 1.

As a preliminary step to prove Proposition \ref{m=1}, we need the following result that starting from any vertex $v$, the propagation process at step $T$ (where $T\leq K_0\ln n$) has the following property.
\begin{proposition}\label{sizesmall}
If $T\leq K_0\ln n$, then for any constant $c>0$, there exists a constant $K_1(c)$ so that any propagation process has the following property
\begin{align*}
\mathbf{Pr}[|\mathcal{Y}_T|+|\mathcal{D}_T|>K_1(c)\ln n]<n^{-c}.
\end{align*}
\end{proposition}

\begin{proof}
Suppose we start from a vertex $v$, which has the following Markov additive process:
\begin{equation*}
   \left\{
 \begin{aligned}
 &Y_0=1,\\
 &Y_{t+1}=Y_{t}+Z_t-1,
 \end{aligned}
 \right.
 \end{equation*}
 where the conditional distribution of $Z_t$ on $Y_t$ is $$Z_t|Y_t\sim \mathbf{B}[n-t-Y_t,p(t)],\quad p(t)=1-(1-p_2)^(1-p_3)^{t}.$$
We introduce another Markov additive process which corresponds to it:
\begin{equation*}
   \left\{
 \begin{aligned}
 &Y_0^{(1)}=1,\\
 &Y_{t+1}^{(1)}=Y_{t}^{(1)}+Z_t^{(1)}-1,
 \end{aligned}
 \right.
 \end{equation*}
 where $$Z_t^{(1)}|Y_t^{(1)}\sim \mathbf{B}[n,p(t)].$$
Apparently, it holds that $\mathbf{B}[n,p(t)]\succ \mathbf{B}[n-t-Y_t,p(t)]$. Take a sequence $l_1=...=l_{T-1}=0,l_T=K_1(c)\ln n-T$, then by Lemma \ref{lemmadayu} we have
\begin{align*}
\mathbf{Pr}[|\mathcal{Y}_T|+|\mathcal{D}_T|>K_1(c)\ln n]&=\mathbf{Pr}[Y_1>0,...,Y_{T-1}>0,Y_T>K_1(c)\ln n-T]\\
&\leq\mathbf{Pr}[Y_1^{(1)}>0,...,Y_{T-1}^{(1)}>0,Y_T^{(1)}>K_1(c)\ln n-T]\\
&\leq\mathbf{Pr}[Y_T^{(1)}>K_1(c)\ln n-T]\\
&=\mathbf{Pr}[Z^{(1)}\equiv Z_1^{(1)}+Z_2^{(1)}+...+Z_{T}^{(1)}>K_1(c)\ln n].
\end{align*}

The generating function of $Z^{(1)}$ is
\begin{align*}
\mathbf{E}[s^{Z^{(1)}}]&=\prod_{t=1}^T\big(1+p(t)(s-1)\big)^n\leq\exp\big[n(s-1)\sum_{t=1}^T p(t)\big]\\
&=\exp\Big[(s-1)\big((1-\epsilon)K_0+cK_0^2/2\big)\ln n\Big]=n^{(s-1)\big((1-\epsilon)K_0+cK_0^2/2\big)}.
\end{align*}
Now we have
\begin{align*}
\mathbf{Pr}[Z^{(1)}>K_1(c)\ln n]\leq\frac{\mathbf{E}[s^{Z^{(1)}}]}{s^{K_1(c)}}\leq\frac{n^{(s-1)((1-\epsilon)K_0+cK_0^2/2)}}{K_1(c)}.
\end{align*}
Take $s=e$, and let $K_1(c)=(e-1)\big((1-\epsilon)K_0+cK_0^2/2\big)$ yields the desired result that
\begin{align*}
\mathbf{Pr}\big[|\mathcal{Y}_T|+|\mathcal{D}_T|>K_1(c)\ln n\big]=\mathbf{Pr}[Z^{(1)}>K_1(c)\ln n]\leq n^{-c}.
\end{align*}

\end{proof}
\subsubsection{Proof of Proposition \ref{m=1}}

Note that $K_0=\max\{1,\lambda_1^{-1}(10)\}\geq\frac{\epsilon}{r}$, so if suffices to prove
$$\mathbf{Pr}[|C_{v_0}|\geq \overline{\omega}\ln n]\leq n^{I_{\epsilon,r}+o(1)},\text{ where }\overline{\omega}=\frac{\epsilon}{r}.$$

Since
$$\mathbf{Pr}[|C_{v_0}|\geq\overline{\omega}\ln n]=\mathbf{Pr}[Y_t>0,t=0,1,...,\overline{\omega}\ln n],$$
we only need to estimate the probability that $Y_t>0$ holds for all $t\in[0,\overline{\omega}\ln n]$ if $Y_0 = 1$ .

\textbf{Step 1. }Take a large positive integer $L$, we divide the interval $[0,\overline{\omega}\ln n]$ into $L$ intervals: $\Delta_k=[t_{k-1},t_k]\equiv[\omega_{k-1}\ln n,\omega_k\ln n]$, where $\omega_k=\frac kL\overline{\omega}$, $k=1,2,...,L$. We now estimate the following probability on the interval $\Delta_k$.
\begin{align*}
\mathbf{Pr}[\Delta_k]\equiv\mathbf{Pr}\big[\forall t\in[t_{k-1},t_k-1], Y_t>0,Y_{t_k}=y_k|Y_{t_{k-1}}=y_{k-1}\big],
\end{align*}
where $\{y_1,...,y_k\}$ is a random sequence taking nonnegative integers, and $y_0 = 1$ for beginning.

Consider the propagation process on the interval $t\in\Delta_k=[t_{k-1},t_k]$:
\begin{equation*}
   \left\{
 \begin{aligned}
 &Y_{t_{k-1}}=y_{k-1},\\
 &Y_{t+1}=Y_{t}+Z_t-1,
 \end{aligned}
 \right.
 \end{equation*}
 where the conditional distribution of $Z_t$ is
 $$Z_t|Y_t\sim \mathbf{B}[n-t-Y_t,p(t)].$$
We introduce another process which corresponds to it:
\begin{equation*}
   \left\{
 \begin{aligned}
 &Y_{t_{k-1}}^{(2)}=y_{k-1},\\
 &Y_{t+1}^{(2)}=Y_{t}^{(2)}+Z_t^{(2)}-1,
 \end{aligned}
 \right.
 \end{equation*}
 where $$Z_t^{(2)}|Y_t^{(2)}\sim B\left[n,\frac{\lambda(\omega_k)}{n}\right].$$
Since $p(t)\leq\lambda(\omega_k)/n$, thus
$$\mathbf{B}[n-t-Y_t,p(t)]\prec \mathbf{B}\left[n,\frac{\lambda(\omega_k)}{n}\right].$$
Taking $l_{t_{k-1}+1}=...=l_{t_{k}-1}=0,l_{t_{k}}=y_k-1$, and applying lemma \ref{lemmadayu}, we have
\begin{align}\label{prdelta1}
\nonumber\mathbf{Pr}[\Delta_k]&=\mathbf{Pr}\big[\forall t\in[t_{k-1},t_k-1], Y_t>0,Y_{t_k}=y_k|Y_{t_{k-1}}=y_{k-1}\big]\\
\nonumber&\leq\mathbf{Pr}\big[\forall t\in[t_{k-1},t_k-1], Y_t>0,Y_{t_k}\geq y_k|Y_{t_{k-1}}=y_{k-1}\big]\\
\nonumber&\leq\mathbf{Pr}\big[\forall t\in[t_{k-1},t_k-1], Y_t^{(2)}>0,Y_{t_k}^{(2)}\geq y_k|Y_{t_{k-1}}^{(2)}=y_{k-1}\big]\\
&\leq\mathbf{Pr}\big[Y_{t_k}^{(2)}\geq y_k|Y_{t_{k-1}}^{(2)}=y_{k-1}\big].
\end{align}
Denote
$$Z^{(2)}\equiv\sum_{t\in[t_{k-1},t_{k}-1]}Z_t^{(2)}\thicksim B\Big[\frac{\overline{\omega}}{L}n\ln n,\frac{\lambda(\omega_k)}{n}\Big].$$
 The generating function of $Z^{(2)}$ is
\begin{align*}
\mathbf{E}[s^{Z^{(2)}}]=\left(1+\frac{\lambda(\omega_k)}{n}  ( s-1) \right)^{\frac{\overline{\omega}}{L}n\ln n}.
\end{align*}
Denote $$d_k=\frac{y_k-y_{k-1}  }{\overline{\omega}\ln n/L},s_k=\frac{1+d_k}{\lambda(\omega_k)},~~~~ \text{where}~~  k =1,2,\cdots,L .$$

If $d_k\geq \lambda(\omega_k)-1$, then $s_k\geq1$. Substituting $s=s_k$ into the generating function, then we can rewrite the inequality (\ref{prdelta1}) as
\begin{align*}
\ln\mathbf{Pr}[\Delta_k]&\leq\ln\mathbf{Pr}\big[Y_{t_k}^{(2)}\geq y_k|Y_{t_{k-1}}^{(2)}=y_{k-1}\big]
=\ln\mathbf{Pr}\big[Z^{(2)}\geq y_k-y_{k-1}+\frac{\overline{\omega}}{L}\ln n\big]\\
&\leq\ln\frac{\mathbf{E}[s_k^{Z^{(2)}}]}{s_k^{y_k-y_{k-1}+\frac{\overline{\omega}}{L}\ln n}}
\leq(s_k-1)\lambda(\omega_k)\frac{\overline{\omega}}{L}\ln n-\left(y_k-y_{k-1}+\frac{\overline{\omega}}{L}\ln n\right)\ln s_k\\
&=\big(1+d_k-\lambda(\omega_k)\big)\frac{\overline{\omega}}{L}\ln n-(1+d_k)\big[\ln(1+d_k)-\ln \lambda(\omega_k)\big]\frac{\overline{\omega}}{L}\ln n.
\end{align*}

By the additive relations of sequence $\{Y_{t_k}=y_k  \ge 1  \}$, we know that $\{d_k\}$ satisfies
$$d_1\geq0,d_1+d_2\geq0,...,d_1+...+d_L\geq0.$$
For simplicity, define a function of $d_k$:
\begin{equation*}
  M_k(d_k)= \left\{
 \begin{aligned}
 &1+d_k-\lambda(\omega_k)-(1+d_k)[\ln(1+d_k)-\ln \lambda(\omega_k)],\text{ \quad if }d_k\geq \lambda(\omega_k)-1,\\
 &0,\text{ \quad\quad\quad\quad\quad\quad\quad \ \ \quad\quad\quad\quad\quad\quad\quad\quad\quad\quad\quad\quad\quad  if }d_k< \lambda(\omega_k)-1.\\
 \end{aligned}
 \right.
 \end{equation*}
Obviously, $ M_k(d_k)$ decreases with $d_k$. Now we have
\begin{align*}
\ln\mathbf{Pr}[\Delta_k]\leq M_k(d_k)\frac{\overline{\omega}}{L}\ln n.
\end{align*}

\textbf{Step 2. } Now we estimate the probability that $Y_t>0$ holds for any $t=0,..., \overline{\omega}\ln n$ and $Y_t$ takes given values on our interval endpoints.
\begin{align*}
&H(y_1,...,y_L)\\
\equiv&\textbf{Pr}\big[Y_t>0, t=0,1,...,\overline{\omega}\ln n; Y_{t_k}=y_k,k=1,...,L\big]\\
=&1_{y_1>0,...,y_L>0}\prod_{1\leq k\leq L}\mathbf{Pr}\big[\forall t\in[t_{k-1},t_k-1], Y_t>0,Y_{t_k}=y_k|Y_{t_{k-1}}=y_{k-1}\big]\\
=&1_{y_1>0,...,y_L>0}\prod_{1\leq k\leq L}\mathbf{Pr}[\Delta_k].
\end{align*}
Thus
$$\ln H(y_1,...,y_L)\leq\frac{\overline{\omega}}{L}\ln n\sum_{1\leq k\leq L} M_k(d_k).$$
Now we make a minor modification of the sequence $\{d_1,...,d_L\}$. Denote $d_k'=\max\{\lambda(\omega_k)-1,d_k\}$, then $M_k(d_k')=M_k(d_k)$ and $d_k'\geq d_k$, also it holds that
$$d_1'\geq0,d_1'+d_2'\geq0,...,d_1'+...+d_L'\geq0.$$

With this in mind, we can rewrite and decompose the sum of $M_k(d_k)$ as
\begin{align*}
&\sum_{1\leq k\leq L} M_k(d_k)= \sum_{1\leq k\leq L} M_k(d_k')\\
=&\sum_{1\leq k\leq L}
 \Big[1+d_k'-\lambda(\omega_k)-(1+d_k')\big[\ln(1+d_k')-\ln \lambda(\omega_k)\big]\Big]\\
=&\sum_{1\leq k\leq L}\big[1-\lambda(\omega_k)+\ln \lambda(\omega_k)\big]+\sum_{1\leq k\leq L}d_k'\ln \lambda(\omega_k)-\sum_{1\leq k\leq L}\big[(1+d_k')\ln(1+d_k')-d_k'\big].
\end{align*}
Since $\ln \lambda(\omega_k)\leq0$ is a decreasing sequence, we rewrite the second sum above by the Abel transformation, and obtain
\begin{align*}
\sum_{1\leq k\leq L}d_k'\ln \lambda(\omega_k)&=(d_1'+...+d_L')\ln \lambda(\omega_L)+\sum_{1\leq k\leq L-1}(d_1'+...+d_k')[\ln \lambda(\omega_k)-\ln \lambda(\omega_{k+1})]\\
&\leq0.
\end{align*}
On the other hand, $g(x)=(1+x)\ln(1+x)-x$ is a concave function with respect to $x$, thus
\begin{align*}
\sum_{1\leq k\leq L}\big[(1+d_k')\ln(1+d_k')-d_k'\big]=\sum_{1\leq k\leq L}g(d_k')\geq L\cdot g\left(\frac{d_1'+...+d_L'}{L}\right)\geq0.
\end{align*}
As a consequence,
\begin{align*}
\sum_{1\leq k\leq L} M_k(d_k)\leq\sum_{1\leq k\leq L}\big[1-\lambda(\omega_k)+\ln \lambda(\omega_k)\big].
\end{align*}

Consider $I_{\epsilon,r}=\int_{0}^{\overline{\omega}}\big[1-\lambda(\omega)+\ln \lambda(\omega)\big]d\omega$. Under the partition $\Delta_k,k=1,...,L$, the upper Darboux sum of $I_{\epsilon,r}$ is
$$\frac{\overline{\omega}}{L}\sum_{1\leq k\leq L}\big[1-\lambda(\omega)+\ln \lambda(\omega)\big].$$
Note that the integrand above is increasing with $\omega$, thus
\begin{align*}
&\frac{\overline{\omega}}{L}\sum_{1\leq k\leq L}\big[1-\lambda(\omega)+\ln \lambda(\omega)\big]-\int_{0}^{\overline{\omega}}\big[1-\lambda(\omega)+\ln \lambda(\omega)\big]d\omega\\
\leq&\Big[\big(1-\lambda(\overline{\omega})+\ln \lambda(\overline{\omega})\big)-\big(1-\lambda(0)+\ln \lambda(0)\big)\Big]\cdot\max(\omega_k-\omega_{k-1})\\
=&-\big(\epsilon+\ln(1-\epsilon)\big)\frac{\overline{\omega}}{L}.
\end{align*}

We are now ready to estimate $H(y_1,...,y_L).$
\begin{align*}
\ln H(y_1,...,y_L)&\leq \frac{\overline{\omega}}{L}\ln n\sum_{1\leq k\leq L}M_k(d_k)\\
&\leq\frac{\overline{\omega}}{L}\ln n\sum_{1\leq k\leq L}\big[1-\lambda(\omega)+\ln \lambda(\omega)\big]\\
&\leq I_{\epsilon,r}\ln n-\big(\epsilon+\ln(1-\epsilon)\big)\overline{\omega}\ln n/L \\
&=(I_{\epsilon,r}+C/L)\ln n ,
\end{align*}
where $C=-(\epsilon+\ln(1-\epsilon))\overline{\omega}$ is a constant.

As a consequence,
\begin{equation*}
H(y_1,...,y_L)\leq n^{I_{\epsilon,r}+C/L  }.
\end{equation*}

\textbf{Step 3. } Finally, returning to the estimate of $\mathbf{Pr}[|C_{v_0}|\geq\overline{\omega}\ln n]$.
\begin{align*}
\mathbf{Pr}&[|C_{v_0}|\geq\overline{\omega}\ln n]=\mathbf{Pr}[Y_t>0,\forall t=0,1,..,\overline{\omega}\ln n]\\
&=\sum_{y_1>0,...,y_L>0}\textbf{Pr}[Y_t>0, t=0,1,...,\overline{\omega}\ln n; Y_{t_k}=y_k,k=1,...,L]\\
&=\sum_{y_1>0,...,y_L>0}H(y_1,...,y_L).
\end{align*}

From Proposition \ref{sizesmall} and take $c=2-I_{\epsilon,r}$, then
\begin{align*}
&\sum_{y_1>0,...,y_j>K_1(c)\ln n,...,y_L>0}H(y_1,...,y_L)\\
\leq&\sum_{y_1>0,...,y_j>K_1(c)\ln n,...,y_L>0}\mathbf{Pr}[Y_{t_k}=y_k,k=1,...,L]=\mathbf{Pr}[Y_{t_j}>K_1(c)\ln n]\\
\leq& n^{-c}=n^{I_{\epsilon,r}-2}.
\end{align*}

Consequently,
\begin{align*}
&\mathbf{Pr}[|C_{v_0}|\geq\overline{\omega}\ln n]\\
=&\sum_{y_1\leq K_1(c)\ln n,...,y_L\leq K_1(c)\ln n}H(y_1,...,y_L)+\sum_{j=1}^{L}\sum_{y_1>0,...,y_j>K_1(c)\ln n,...,y_L>0}H(y_1,...,y_L)\\
\leq& (K_1(c)\ln n)^Ln^{I_{\epsilon,r}}n^{C/L}+Ln^{I_{\epsilon,r}-2}\\
\leq &n^{I_{\epsilon,r}+\varepsilon},
\end{align*}
where the last inequality holds for any sufficiently small constant $\varepsilon=\varepsilon(L)>0$.

\subsubsection{Special case of $p_2=0$}
In our previous analysis, all the propagation processes start from exactly one random variable. If we consider a special case of $p_2=0$, then the hypergraph $\mathbb{G}=\mathbb{G}(n,p_2,p_3)$ is exactly the 3-uniform hypergraph. In this case, we have to start the propagation process with two vertices instead of just one. Similarly we denote by $C_{\{v_1,v_2\}}$ the component generated by starting from $v_1,v_2$, where the propagation process stays the same except for the initial set $\mathcal{Y}_0$:
\begin{equation*}
   \left\{
 \begin{aligned}
 &Y_0=2,\\
 &Y_{t+1}=Y_{t}+Z_t-1.
 \end{aligned}
 \right.
 \end{equation*}
Similar to Proposition \ref{m=1}, propagation processes which start from any two arbitrary vertices will terminate w.h.p. before $O(\ln n)$ steps.
\begin{corollary}\label{m2}
Let $v_1,v_2$ be any random vertices, then
\begin{equation*}
\mathbf{Pr}[|C_{v_1,v_2}|\geq K_0\ln n]\leq n^{I_{\epsilon,r}+o(1)}.
\end{equation*}
\end{corollary}
The proof of Corollary \ref{m2} follows the technique of Proposition \ref{m=1}. Because very minor modifications are needed, we move the proof to the appendix.

From Corollary \ref{m2}, if we let $I_{\epsilon,r}<-2$, then
\begin{equation*}
\mathbf{Pr}[\exists\text{ two vertices }v_1,v_2,\text{ s.t. } |C_{v_1,v_2}|\geq K_0\ln n]=\binom n2\cdot n^{I_{\epsilon,r}+o(1)}=o(1),
\end{equation*}
which implies that if  $I_{\epsilon,r}<-2$, then w.h.p. $\mathbb{G}$ is not propagation connected.

Recall the results in Theorem \ref{COW} obtained by \cite{COW}, and note that if we let $p_2=0$ (i.e. $\epsilon=1$), then the condition $I_{\epsilon,r}<-2$ is equivalent to $r<0.25$. Now we can say that if $r<0.25$, then $\mathbf{H}(n,p)$ fails to be propagation connected w.h.p., which shows that the propagation connectivity of 3-uniform hypergraph has a sharp phase transition. Therefore our results close the gap left by \cite{COW}.

\subsection{Supercritical regime: $I_{\epsilon,r}>-1$}

From previous results, we know that once the propagation process of a vertex survives up to $O(\ln n)$ steps, it may be possible to survive until the end, i.e., the entire hypergraph $\mathbb{G}$
can be propagation connected. This leads us to the definition of ``good'' vertices.
\begin{definition}
We call a vertex $v$ good, if the propagation process starting from $v$ survives at least $K_0\ln n$ steps, i.e. $|C_v|> K_0\ln n$.
\end{definition}

\begin{proposition}\label{progreater}
For a random vertex $v$, we  have
$$\mathbf{Pr}[|C_v|\geq K_0\ln n]\geq n^{I_{\epsilon,r}-o(1)}.$$
\end{proposition}

\begin{proof}

To estimate the probabilisty that $Y_t>0$ holds for all $t\in[0,K_0\ln n]$, we first consider some interval $t\in[a,b]\subset[0,K_0\ln n]$. Suppose at the beginning $Y_a=y>0$, and the propagation process is
\begin{equation*}
   \left\{
 \begin{aligned}
 &Y_a=y,\\
 &Y_{t+1}=Y_{t}+Z_t-1,
 \end{aligned}
 \right.
 \end{equation*}
 where $$Z_t|Y_t\sim \mathbf{B}[n-t-Y_t,p(t)],\quad p(t)=1-(1-p_2)(1-p_3)^t.$$
We define another Markov process which corresponds to it as
\begin{equation*}
   \left\{
 \begin{aligned}
 &Y_a^{(3)}=y,\\
 &Y_{t+1}^{(3)}=Y_{t}^{(3)}+Z_t^{(3)}-1,
 \end{aligned}
 \right.
 \end{equation*}
 where $$Z_t^{(3)}|Y_t^{(3)}\sim \mathbf{B}[n-b,p(a)].$$

By lemma \ref{lemmadayu}, $\mathbf{B}[n-t-Y_t,p(t)]\succ \mathbf{B}[n-b,p(a)]$ holds for all $Y_t\leq b-t$, where we take $M=b-a\geq y,q=1,l_a=l_{a+1}=...=l_b=0$. Therefore
\begin{align*}
\mathbf{Pr}\big[\forall t\in[a,b],Y_t>0|Y_a=y\big]
&\geq\mathbf{Pr}\big[\forall t\in[a,b],Y_t^{(3)}>0|Y_a^{(3)}=y\big]\\
&\geq\mathbf{Pr}[z^{(3)}\geq b-a],
\end{align*}
where
$Z^{(3)}\equiv Y_b^{(3)}-Y_a^{(3)}+b-a=Z_a^{(3)}+...+Z_{b-1}^{(3)}$. Thus
\begin{align*}
Z^{(3)}\thicksim \mathbf{B}[(n-b)(b-a),p(a)].
\end{align*}
As a result,
\begin{align*}
\mathbf{Pr}[Z^{(3)}=b-a]=\binom{(n-b)(b-a)}{b-a}p(a)^{b-a}\big(1-p(a)\big)^{(n-b-1)(b-a)}.
\end{align*}

Note that
\begin{align*}
\ln\binom{(n-b)(b-a)}{b-a}
\geq&\ln\frac{[(n-b-a)(b-a)]^{b-a}}{(b-a)!}\\
>&(b-a)\big[\ln(n-b-1)+1\big]-\ln(b-a)-2.
\end{align*}
Now  we have
\begin{align*}
&\ln\mathbf{Pr}[Z^{(3)}=b-a]\\
\geq&
(b-a)\big[\ln(n-b-1)+1+\ln p(a)-(n-b-1)p(a)\big]-\ln(b-a)-2\\
=&(b-a)\Big(1+\ln\big((n-b-a)p(a)\big)-(n-b-a)p(a)\Big)-\ln(b-a)-2.
\end{align*}

Similarly, we consider the partition of the interval $[0,\overline{\omega}\ln n]$: $\Delta_k=[\omega_{k-1}\ln n,\omega_k\ln n],k=1,...,L$. Take $[a,b]=\Delta_k$, then $(n-b-1)p(a)=(1+o(1))\lambda(\omega_{k-1})$.

Since random variables $Z_a^{(3)}-1,...,Z_{b-1}^{(3)}-1$ are identically distributed, by lemma \ref{lemmadayu},
\begin{align*}
&\mathbf{Pr}\big[\forall t\in[a,b],Y_t>0|Y_a=y\big]\geq\mathbf{Pr}\big[\forall t\in[a,b],Y_t^{(3)}>0|Y_a^{(3)}=y\big]\\
\geq&\mathbf{Pr}\big[\forall k=0,1,...,b-a-1, Z_a^{(3)}-1+...+Z_{a+k}^{(3)}-1\geq0\big]\\
\geq&\frac{1}{b-a}\mathbf{Pr}[Z^{(3)}\geq b-a].
\end{align*}
Therefore,
\begin{align*}
&\mathbf{Pr}\big[\forall t\in[a,b],Y_t>0|Y_a>0\big]\geq\frac{1}{b-a}\mathbf{Pr}[Z^{(3)}\geq b-a]\\
\geq &(b-a)\Big(1+\ln\big((n-b-a)p(a)\big)-(n-b-a)p(a)\Big)-2\ln(b-a)-2.
\end{align*}
Consequently,
\begin{align*}
&\ln\mathbf{Pr}\big[\forall t\in[0,\overline{\omega}\ln n\big], Y_t>0|Y_0=1]=\sum_{\Delta_k}\ln\mathbf{Pr}[\forall t\in \Delta_k,Y_t>0|Y_0=1]\\
\geq&\sum_{\Delta_k}(b-a)\Big(1+\ln\big[(n-b-a)p(a)\big]-(n-b-a)p(a)\Big)-2\ln(b-a)-2\\
\geq&\ln n\sum_{\Delta_k}(\omega_k-\omega_{k-1})\big[1+\ln \lambda(\omega_{k-1})-\lambda(\omega_{k})+o(1)\big]-O(L\ln\ln n)\\
=&\ln n\sum_{\Delta_k}(\omega_k-\omega_{k-1})\big[1+\ln \lambda(\omega_{k-1})-\lambda(\omega_{k})\big]-o(\overline{\omega}\ln n)-O(L\ln\ln n),
\end{align*}
where the term $o(1)$ is of order $O(1/\ln n)$.

Note that $\sum_{\Delta_k}(\omega_k-\omega_{k-1})\big(1+\ln \lambda(\omega_{k-1})-\lambda(\omega_{k})\big)$ is a lower Darboux sum of the integral $I_{\epsilon,r}=\int_{0}^{\overline{\omega}}\big[1-\lambda(\omega)+\ln \lambda(\omega)\big]d\omega$ under the partition $\Delta_k=[\omega_{k-1}\ln n,\omega_k\ln n],k=1,...,L$. Also note that the integrand is increasing with $\omega$, thus
\begin{align*}
I_{\epsilon,r}-&\sum_{\Delta_k}(\omega_k-\omega_{k-1})\big(1+\ln \lambda(\omega_{k-1})-\lambda(\omega_{k})\big)\\
\leq&\Big[\big(1-\lambda(\omega)+\ln \lambda(\omega)\big)-\big(1-\lambda(0)+\ln \lambda(0)\big)\Big]\cdot\max(\omega_k-\omega_{k-1})\\
=&-\big(\epsilon+\ln(1-\epsilon)\big)\frac{\overline{\omega}}{L}.
\end{align*}
Now we have
\begin{align*}
&\ln\mathbf{Pr}\big[\forall t\in[0,\overline{\omega}\ln n], Y_t>0|Y_0=1\big]-I_{\epsilon,r}\ln n\\
\geq&-\big(\epsilon+\ln(1-\epsilon)\big)\frac{\overline{\omega}}{L}\ln n-o(\overline{\omega}\ln n)-O(L\ln\ln n)=-o(\ln n).
\end{align*}
\\

Next we take another interval $[a',b']\equiv[\overline{\omega}\ln n,K_0\ln n]$, then $(n-b'-1)p(a')=1-o(1)$.
\begin{align*}
&\ln\mathbf{Pr}[\forall t\in[a',b'], Y_t>0|Y_{a'}>0]\\
\geq&(b'-a')\Big(1+\ln\big((n-b'-a')p(a')\big)-(n-b'-a')p(a')\Big)-2\ln(b'-a')-2\\
=&(b'-a')o(1)-2\ln(b'-a')-2=-o(\ln n).
\end{align*}
Consequently,
\begin{align*}
\ln\mathbf{Pr}\big[\forall t\in[0,\overline{\omega}\ln n], Y_t>0|Y_0=1\big]-I_{\epsilon,r}\ln n
\geq-o(\ln n),
\end{align*}
which leads to
\begin{align*}
\mathbf{Pr}\big[|C_v|\geq K_0\ln n\big]\geq n^{I_{\epsilon,r}-o(1)}.
\end{align*}
\end{proof}

With Proposition \ref{progreater} in mind, and note that $I_{\epsilon,r}>-1$ is a constant, it is straightforward to get the following result.
\begin{proposition}
Let $X$ be the number of good vertices. There exists a constant $\varepsilon>0$, such that $$\mathbf{E}[X]>n^\varepsilon.$$
\end{proposition}

The following result guarantees the existence of ``good'' vertices, which implies that random hypergraph $\mathbb{G}$ can be propagation connected.
\begin{lemma}
Let $X$ be the number of good vertices. Then
 $$\mathbf{Pr}[X>0]=1-o(1).$$
\end{lemma}

\begin{proof}

First we estimate the probability that vertices $v_0$ and $v_1$ are both good. To start with, we consider the conditional probability $\mathbf{Pr}[v_1\text{ is good}|v_0\text{ is good}]$.

Assume that $\{\mathcal{Y}_t,\mathcal{D}_t\}$ and $\{\mathcal{Y}_t',\mathcal{D}_t'\}$ represent two propagation processes starting from $v_0$ and $v_1$, respectively. Recall the constant $K_0$ which is defined in Lemma \ref{middlesize}, then during the first $K_0\ln n$ steps, there are two cases:

$\mathbf{(1) }$ We denote by $B$ the event that there exists some $t\leq K_0\ln n$ such that two sets $\mathcal{Y}_{t}'\bigcup \mathcal{D}_{t}'$ and $\mathcal{Y}_{K_0\ln n}\bigcup \mathcal{D}_{K_0\ln n}$ have common vertices, i.e., $$\big(\mathcal{Y}_{t}'\bigcup \mathcal{D}_{t}'\big)\bigcap\big(\mathcal{Y}_{K_0\ln n}\bigcup \mathcal{D}_{K_0\ln n}\big) \neq\emptyset.$$

$\mathbf{(2) }$ Denote by $\overline{B}$ the event that, for any $t\leq K_0\ln n$, the two sets $\mathcal{Y}_{t}'\bigcup \mathcal{D}_{t}'$ and  $\mathcal{Y}_{K_0\ln n}\bigcup \mathcal{D}_{K_0\ln n}$ have no common vertices, i.e.,  $$\big(\mathcal{Y}_{t}'\bigcup \mathcal{D}_{t}'\big)\bigcap\big(\mathcal{Y}_{K_0\ln n}\bigcup \mathcal{D}_{K_0\ln n}\big)=\emptyset.$$\\

Firstly, we consider the case (1).

Denote  $R=\mathcal{Y}_{K_0\ln n}\cup\mathcal{D}_{K_0\ln n}$ the set of vertices connected from $v_0$. Then by Proposition \ref{sizesmall}, for $c=2-I_{\epsilon,r}$ there exists a constant $K_1(c)$ such that $\mathbf{Pr}[|R|>K_1(c)\ln n]<n^{-c}$. Let $B_t$ be the event that the two propagation processes starting from $v_0$ and $v_1$ do not intersect in the first $t-1$ steps, but intersect at the $t$-th step. In fact, event $B$ is the union of all $B_t$ ($t=1,...,K_0\ln n$).
\begin{align}
&\nonumber\mathbf{Pr}\big[v_1\text{ is good}, B\big|v_0\text{ is good}\big]\leq\mathbf{Pr}\big[ B\big|v_0\text{ is good}\big]\\
\nonumber\leq&\mathbf{Pr}\big[B,|R|>K_1(c)\ln n\big|v_0\text{ is good}\big]+\mathbf{Pr}\big[B,|R|\leq K_1(c)\ln n,\big|v_0\text{ is good}\big]\\
\nonumber\leq&\mathbf{Pr}\big[|R|>K_1(c)\ln n\big|v_0\text{ is good}\big]+\mathbf{Pr}\big[B\big||R|\leq K_1(c)\ln n,v_0\text{ is good}\big]\\
\nonumber\leq&\frac{n^{I_{\epsilon,r}-2}}{\mathbf{Pr}[v_0\text{ is good}]}+\sum_{t=0}^{K_0\ln n}\mathbf{Pr}\big[B_t\big||R|\leq K_1(c)\ln n,v_0\text{ is good}\big]\\
\nonumber\leq&n^{-2+\varepsilon}+\sum_{t=0}^{K_0\ln n}|R|\cdot p(t)\leq n^{-2+\varepsilon}+\sum_{t=0}^{K_0\ln n}O[\ln n/n]\\
=&O\big((\ln n)^2/n\big),
\end{align}
where the constant terms implicit in $O$ is at most $K_0\cdot K_1(c)\cdot \lambda(K_0)$.
The second to last inequality is because that, conditioned on the information of the propagation process starting from $v_0$, the event $B_t$ is equivalent to that the process from $v_1$ connects vertices in $R$ at step $t$. Recall that the probability for a vertex to be connected by a process at time $t$ is $p(t)$, and there are $|R|$ vertices in $R$, thus  the inequality follows naturally.
\\

Secondly, we consider the case (2).

Since the two sets $\mathcal{Y}_{K_0\ln n}\bigcup \mathcal{D}_{K_0\ln n}$ and $\mathcal{Y}_{K_0\ln n}'\bigcup \mathcal{D}_{K_0\ln n}'$ have no intersections, then the first $K_0\ln n$ steps of the propagation process which starts from vertex $v_1$ contains no vertices from the set $\mathcal{Y}_{K_0\ln n}\bigcup \mathcal{D}_{K_0\ln n}$. In fact in this case, $v_1$ has the following Markov additive process:
\begin{equation*}
   \left\{
 \begin{aligned}
 &Y_0'=1,\\
 &Y_{t+1}'=Y_{t}'+Z_t'-1,
 \end{aligned}
 \right.
 \end{equation*}
where \begin{equation}\label{Ztdis2}
Z_t'|Y_{t}'\sim \mathbf{B}[n-t-|R|-Y_{t}',p(t)],p(t)=1-(1-p_2)(1-p_3)^{t}.
\end{equation}

Note that $\mathbf{B}[n-t-z,p(t)]\succ \mathbf{B}[n-t-|R|-z,p(t)]$ holds for all $z$.
Applying lemma \ref{lemmadayu}, if $t=K_0\ln n$, then we have
 \begin{align}
 \mathbf{Pr}[Y_t'\geq 0, Y_1'>0,...,Y_{t-1}'>0]\leq \mathbf{Pr}[Y_t\geq 0, Y_1>0,...,Y_{t-1}>0].
 \end{align}
This means that, the probability of the propagation process starting from $v_1$ survives up to time $K_0\ln n$ is less than that of $v_0$, i.e.,
\begin{align}
\mathbf{Pr}[v_1\text{ is good}, \overline{B}|v_0\text{ is good}]\leq\mathbf{Pr}[v_1\text{ is good}|\overline{B},v_0\text{ is good}]\leq \mathbf{Pr}[v_1\text{ is good}]= \mathbf{Pr}[v_0\text{ is good}].
\end{align}
Therefore, we combine case (1) and case (2) to get
\begin{align}
\nonumber\mathbf{Pr}[v_1\text{ is good}|v_0\text{ is good}]&=\mathbf{Pr}[v_1\text{ is good}, B|v_0\text{ is good}]+\mathbf{Pr}[v_1\text{ is good}, \overline{B}|v_0\text{ is good}]\\
\nonumber&\leq O\big((\ln n)^2/n\big)+\mathbf{Pr}[v_0\text{ is good}]\\
&=(1+o(1))\mathbf{Pr}[v_0\text{ is good}].
\end{align}
Recall that $\mathbf{Pr}[v_0\text{ is good}]>n^{-1+\varepsilon}$, as a result,
\begin{align}
\nonumber\mathbf{E}[X^2]&=n(n-1)\mathbf{Pr}[v_1\text{ is good},v_0\text{ is good}]+n\mathbf{Pr}[v_0\text{ is good}]\\
\nonumber&=\mathbf{Pr}[v_0\text{ is good}]\Big(n+n(n-1)\mathbf{Pr}[v_1\text{ is good}|v_0\text{ is good}]\Big)\\
\nonumber&=\mathbf{Pr}[v_0\text{ is good}]\big(n+(1+o(1))n(n-1)\mathbf{Pr}[v_0\text{ is good}]\big)\\
&=(1+o(1))n^2\mathbf{Pr}[v_0\text{ is good}]^2=(1+o(1))\mathbf{E}[X]^2.
\end{align}

Finally, by the Cauchy inequality, we have
$$\mathbf{Pr}[X>0]\geq\frac{\mathbf{E}[X]^2}{\mathbf{E}[X^2]}=1+o(1).$$

\end{proof}
\section{Conclusion}

In this paper, we study the propagation connectivity of a special random hypergraph $\mathbb{G}(n,p_2,p_3)$ which is composed of edges of length 2 and hyperedges of length 3. Using probabilistic arguments, we prove that there is a sharp phase transition of propagation connectivity of $\mathbb{G}(n,p_2,p_3)$ that if $I_{\epsilon,r}<-1$ ($I_{\epsilon,r}$ is a constant defined by the parameters $\epsilon$ and $r$), then w.h.p. there is no vertex that propagation connects all the vertices in $\mathbb{G}$, moreover, every propagation connected component in $\mathbb{G}$ contains no more than $O(\ln n)$ vertices; while if $I_{\epsilon,r}>-1$, then w.h.p. there exist ``good" vertices which can surprisingly connect all vertices in the entire graph  once they can propagation connect more that $O(\ln n)$ vertices.

In view of the possible relevant connection with the discrete, realistic and complex networks,
we believe the above results would be useful to provide a theoretical foundation for our understanding of various complex networks that permeate this information age.\\

\noindent\textbf{Appendix. Proof of Corollary \ref{m2}.}

The proof of Corollary \ref{m2} is a minor modification of Proposition \ref{m=1}.

Note that $K_0=\max\{1,\lambda_1^{-1}(10)\}\geq\frac{\epsilon}{r}$, so we only need to prove
$$\mathbf{Pr}[|C_{\{v_1,v_2\}}|\geq \overline{\omega}\ln n]\leq n^{I_{\epsilon,r}+o(1)},\text{ where }\overline{\omega}=\frac{\epsilon}{r}.$$
Since
$$\mathbf{Pr}[|C_{\{v_1,v_2\}}|\geq\overline{\omega}\ln n]=\mathbf{Pr}[Y_t>0,t=0,1,...,\overline{\omega}\ln n],$$
we only need to estimate the probability that $Y_t>0$ holds for all $t\in[0,\overline{\omega}\ln n]$ if $Y_0 = m$ .

\textbf{Step 1. }Take a large positive integer $L$, we divide the interval $[0,\overline{\omega}\ln n]$ into $L$ intervals: $\Delta_k=[t_{k-1},t_k]\equiv[\omega_{k-1}\ln n,\omega_k\ln n]$, where $\omega_k=\frac kL\overline{\omega}$, $k=1,2,...,L$. We now estimate the following probability on the interval $\Delta_k$.
\begin{align*}
\mathbf{Pr}[\Delta_k]\equiv\mathbf{Pr}\big[\forall t\in[t_{k-1},t_k-1], Y_t>0,Y_{t_k}=y_k|Y_{t_{k-1}}=y_{k-1}\big],
\end{align*}
where $\{y_1,...,y_k\}$ is a random sequence taking nonnegative integers, and $y_0 = 2$ for beginning.

Consider the propagation process on the interval $t\in\Delta_k=[t_{k-1},t_k]$:
\begin{equation*}
   \left\{
 \begin{aligned}
 &Y_{t_{k-1}}=y_{k-1},\\
 &Y_{t+1}=Y_{t}+Z_t-1,
 \end{aligned}
 \right.
 \end{equation*}
 where $$Z_t|Y_t\sim \mathbf{B}[n-t-Y_t,p(t)],\text{ where }p(t)=1-(1-p_2)(1-p_3)^t.$$
Define another process which corresponds to it as:
\begin{equation*}
   \left\{
 \begin{aligned}
 &Y_{t_{k-1}}^{(2)}=y_{k-1},\\
 &Y_{t+1}^{(2)}=Y_{t}^{(2)}+Z_t^{(2)}-1,
 \end{aligned}
 \right.
 \end{equation*}
 where (recall (\ref{funlamda})) $$Z_t^{(2)}|Y_t^{(2)}\sim B\left[n,\frac{\lambda(\omega_k)}{n}\right].$$
Since $p(t)\leq\lambda(\omega_k)/n$, thus
$$\mathbf{B}[n-t-Y_t,p(t)]\prec \mathbf{B}\left[n,\frac{\lambda(\omega_k)}{n}\right].$$
Taking $l_{t_{k-1}+1}=...=l_{t_{k}-1}=0,l_{t_{k}}=y_k-1$, and applying lemma \ref{lemmadayu}, we have
\begin{align}\label{prdelta}
\nonumber\mathbf{Pr}[\Delta_k]&=\mathbf{Pr}\big[\forall t\in[t_{k-1},t_k-1], Y_t>0,Y_{t_k}=y_k|Y_{t_{k-1}}=y_{k-1}\big]\\
\nonumber&\leq\mathbf{Pr}\big[\forall t\in[t_{k-1},t_k-1], Y_t>0,Y_{t_k}\geq y_k|Y_{t_{k-1}}=y_{k-1}\big]\\
\nonumber&\leq\mathbf{Pr}\big[\forall t\in[t_{k-1},t_k-1], Y_t^{(2)}>0,Y_{t_k}^{(2)}\geq y_k|Y_{t_{k-1}}^{(2)}=y_{k-1}\big]\\
&\leq\mathbf{Pr}\big[Y_{t_k}^{(2)}\geq y_k|Y_{t_{k-1}}^{(2)}=y_{k-1}\big].
\end{align}
Denote
$$Z^{(2)}\equiv\sum_{t\in[t_{k-1},t_{k}-1]}Z_t^{(2)}\thicksim B\Big[\frac{\overline{\omega}}{L}n\ln n,\frac{\lambda(\omega_k)}{n}\Big].$$
 The generating function of $Z^{(2)}$ is
\begin{align*}
\mathbf{E}[s^{Z^{(2)}}]=\left(1+\frac{\lambda(\omega_k)}{n}  ( s-1) \right)^{\frac{\overline{\omega}}{L}n\ln n}.
\end{align*}
Denote $$d_k=\frac{y_k-y_{k-1}  }{\overline{\omega}\ln n/L},s_k=\frac{1+d_k}{\lambda(\omega_k)}~~~~ \text{for}~~  k =2,\cdots,L .$$
and when $ k=1 $  denote  $ d_1 =\frac{y_1-y_0 +1  }{\overline{\omega}\ln n/L} = \frac{y_1-1   }{\overline{\omega}\ln n/L} ,s_1=\frac{1+d_1}{\lambda(\omega_1)}. $

If $d_k\geq \lambda(\omega_k)-1$, then $s_k\geq1$. Substitute $s=s_k$ into the generating function, now we can rewrite the inequality (\ref{prdelta}) as (when $k\ge 2$)
\begin{align*}
\ln\mathbf{Pr}[\Delta_k]&\leq\ln\mathbf{Pr}\big[Y_{t_k}^{(2)}\geq y_k|Y_{t_{k-1}}^{(2)}=y_{k-1}\big]
=\ln\mathbf{Pr}\big[Z^{(2)}\geq y_k-y_{k-1}+\frac{\overline{\omega}}{L}\ln n\big]\\
&\leq\ln\frac{\mathbf{E}[s_k^{Z^{(2)}}]}{s_k^{y_k-y_{k-1}+\frac{\overline{\omega}}{L}\ln n}}
\leq(s_k-1)\lambda(\omega_k)\frac{\overline{\omega}}{L}\ln n-\left(y_k-y_{k-1}+\frac{\overline{\omega}}{L}\ln n\right)\ln s_k\\
&=\big(1+d_k-\lambda(\omega_k)\big)\frac{\overline{\omega}}{L}\ln n-(1+d_k)\big[\ln(1+d_k)-\ln \lambda(\omega_k)\big]\frac{\overline{\omega}}{L}\ln n.
\end{align*}
and (when $k=1$)
\begin{align*}
\ln\mathbf{Pr}[\Delta_1]&
\leq(s_1 -1)\lambda(\omega_1 )\frac{\overline{\omega}}{L}\ln n-\left(y_1-y_0+\frac{\overline{\omega}}{L}\ln n\right)\ln s_1 \\
& = (s_1 -1)\lambda(\omega_1 )\frac{\overline{\omega}}{L}\ln n-\left(y_1-y_0+(m-1)+\frac{\overline{\omega}}{L}\ln n\right)\ln s_1 +\ln s_1 \\
& = \big(1+d_1-\lambda(\omega_1)\big)\frac{\overline{\omega}}{L}\ln n-(1+d_1)\big[\ln(1+d_1)-\ln \lambda(\omega_1)\big]\frac{\overline{\omega}}{L}\ln n + \ln s_1.
\end{align*}

Since $\lambda(\omega_1)= 1 - \epsilon + r \omega_1 = 1 - \epsilon +  r \frac{\epsilon }{rL} \ge \frac{1}{L} $ . And we assume that $y_k \le K_1(c) \ln n$ when we use the inequality for $ \ln\mathbf{Pr}[\Delta_k] $, which implies that $d_1 \le \frac{L K_1(c)} {\overline{\omega} } = \frac{L K_1(c) r} { \epsilon  } $. So $ s_1  = \frac{1+d_1}{\lambda(\omega_1)} \le L^2 \times \text{const}$ , and
 $$(m-1) \ln s_1  \le 2(m-1) \ln L + \text{const} = O(\ln L) .$$
By the additive relations of sequence $\{Y_{t_k}=y_k  \ge 1  \}$, we know that $\{d_k\}$ satisfies
$$d_1\geq0,d_1+d_2\geq0,...,d_1+...+d_L\geq0.$$
For simplicity, define a function of $d_k$:
\begin{equation*}
  M_k(d_k)= \left\{
 \begin{aligned}
 &1+d_k-\lambda(\omega_k)-(1+d_k)[\ln(1+d_k)-\ln \lambda(\omega_k)],\text{ \quad if }d_k\geq \lambda(\omega_k)-1,\\
 &0,\text{ \quad\quad\quad\quad\quad\quad\quad \ \ \quad\quad\quad\quad\quad\quad\quad\quad\quad\quad\quad\quad\quad  if }d_k< \lambda(\omega_k)-1.\\
 \end{aligned}
 \right.
 \end{equation*}
Obviously, $ M_k(d_k)$ decreases with $d_k$. Now we have
\begin{align*}
\ln\mathbf{Pr}[\Delta_k]\leq M_k(d_k)\frac{\overline{\omega}}{L}\ln n + O(\ln L)  \mathbf{1} _{k=1} .
\end{align*}

\textbf{Step 2. } Since step 2 is exactly the same with Proposition \ref{m=1}, we omit it to avoid duplication.

\textbf{Step 3. } Finally, returning to the estimate of $\mathbf{Pr}[|C_{\{v_1,v_2\}}|\geq\overline{\omega}\ln n]$.
\begin{align*}
\mathbf{Pr}&[|C_{\{v_1,v_2\}}|\geq\overline{\omega}\ln n]=\mathbf{Pr}[Y_t>0,\forall t=0,1,..,\overline{\omega}\ln n]\\
&=\sum_{y_1>0,...,y_L>0}\textbf{Pr}[Y_t>0, t=0,1,...,\overline{\omega}\ln n; Y_{t_k}=y_k,k=1,...,L]\\
&=\sum_{y_1>0,...,y_L>0}H(y_1,...,y_L).
\end{align*}

Recall Proposition \ref{sizesmall} and take $c=2-I_{\epsilon,r}$, then
\begin{align*}
&\sum_{y_1>0,...,y_j>K_1(c)\ln n,...,y_L>0}H(y_1,...,y_L)\\
\leq&\sum_{y_1>0,...,y_j>K_1(c)\ln n,...,y_L>0}\mathbf{Pr}[Y_{t_k}=y_k,k=1,...,L]=\mathbf{Pr}[Y_{t_j}>K_1(c)\ln n]\\
\leq& n^{-c}=n^{I_{\epsilon,r}-2}.
\end{align*}

Consequently,
\begin{align*}
&\mathbf{Pr}[|C_{\{v_1,v_2\}}|\geq\overline{\omega}\ln n]\\
=&\sum_{y_1\leq K_1(c)\ln n,...,y_L\leq K_1(c)\ln n}H(y_1,...,y_L)+\sum_{j=1}^{L}\sum_{y_1>0,...,y_j>K_1(c)\ln n,...,y_L>0}H(y_1,...,y_L)\\
\leq& L^{O(1)} (K_1(c)\ln n)^Ln^{I_{\epsilon,r}}n^{C/L}+Ln^{I_{\epsilon,r}-2}\\
\leq &n^{I_{\epsilon,r}+\varepsilon},
\end{align*}
where the last inequality holds for any sufficiently small constant $\varepsilon=\varepsilon(L)>0$.

\end{document}